\documentclass[11pt,reqno]{amsart}

\usepackage{amssymb}
\usepackage{amsfonts}

\numberwithin{equation}{section}
     
\theoremstyle{plain}

\newtheorem{theorem}[subsection]{Theorem}
\newtheorem{proposition}[subsection]{Proposition}
\newtheorem{lemma}[subsection]{Lemma}

\theoremstyle{definition}

\newtheorem*{mpet-rpt}{Theorem \ref{mpet}}

\renewcommand{\leq}{\leqslant}
\renewcommand{\geq}{\geqslant}

\newcommand{\md}[1]{\ensuremath{(\operatorname{mod}\, #1)}}

\newcommand\E{{\mathbb{E}}}
\newcommand\Z{\mathbb{Z}}
\newcommand\R{\mathbb{R}}

\newcommand\C{\mathbb{C}}

\newcommand\eps{\varepsilon}
\newcommand\GI{\operatorname{GI}}

\newcommand\diam{\operatorname{diam}}
\newcommand\poly{\operatorname{poly}}

\begin{document}
\title[Szemer\'edi's theorem]{Yet another proof of Szemer\'edi's theorem}
\author{Ben Green}
\address{Centre for Mathematical Sciences\\
Wilberforce Road\\
Cambridge CB3 0WA\\
England }
\email{b.j.green@dpmms.cam.ac.uk}
\author{Terence Tao}
\address{Department of Mathematics\\
UCLA\\
Los Angeles, CA 90095\\
USA}
\email{tao@math.ucla.edu}

\subjclass{}

\begin{abstract} 
Using the density-increment strategy of Roth and Gowers, we derive Szemer\'edi's theorem on arithmetic progressions from the inverse conjectures $\GI(s)$ for the Gowers norms, recently announced by the authors and Ziegler in \cite{uk-inverse}. 
\end{abstract}

\maketitle

\begin{center}
\emph{To Endre Szemer\'edi on the occasion of his 70th birthday.}
\end{center}

\tableofcontents

\setcounter{tocdepth}{1}

\section{Introduction}

In this note we show how Szemer\'edi's famous theorem \cite{szemeredi-4,szemeredi-aps} on arithmetic progressions follows from the inverse conjecture $\GI(s)$ for the Gowers norms, recently announced in \cite{uk-inverse}. This paper is designed as a coda to \cite{green-tao-arithregularity}, and in particular we refer the reader to that paper (or to many other places in the literature) for the definition of the following terms, which we shall use without further comment: \emph{filtered nilmanifold of complexity $\leq M$, polynomial sequence, degree $\leq s$ polynomial nilsequence of complexity at most $M$, rational polynomial sequence, Gowers norm, generalised von Neumann theorem} and \emph{smoothness norm $C^{\infty}[N]$}.

Our main point is to show that Szemer\'edi's theorem can actually be derived rather easily from $\GI(s)$. We gave a different deduction in \cite{green-tao-arithregularity}, designed to illustrate that for a large class of theorems (including Szemer\'edi's theorem) it essentially suffices to ``check the result for nilsequences''. That argument was somewhat complicated, not least because it relied heavily on the quantitative distribution results for nilsequences obtained in \cite{green-tao-nilratner}.

The argument we give here is based on the density-increment strategy of Roth \cite{roth} and Gowers \cite{gowers-4aps, gowers-longaps}. In fact our argument is, structurally, the same as that of Gowers except that we use the inverse theorem as a black box rather than prove, as Gowers did, a weaker version of it.

The only remotely new technical result in this note is the following. Here, and elsewhere in the paper, write $\diam_S(f) := \sup_{s_1, s_2 \in S} d_X(f(s_1) ,f(s_2))$ whenever $f : S \rightarrow X$ is some function from a set $S$ into a metric space $(X, d_X)$.

\begin{theorem}\label{recur-thm}
Let $\eps > 0$ be a real parameter, let $s \geq 1$ be an integer, and let $M$ be a complexity parameter. Then there is a number $\kappa_{s,M} > 0$ with the following property. Let $(F(g(n)\Gamma))_{n \in \Z}$ be a degree $\leq s$ polynomial nilsequence of complexity at most $M$, and let $P \subseteq \Z$ be an arithmetic progression. Then we may partition $P$ into a disjoint union of progressions $P_i$, $i = 1,\dots, m$, each of size $\gg_{M,\eps} |P|^{\kappa_{s,M}}$, such that  
\[ \diam_{P_i}(F(g(n)\Gamma)) \leq \eps\] for all $i = 1,\dots, m$.
\end{theorem}
\emph{Remark.} The progressions $P_i$ need not have the same common difference. 

We prove this theorem in \S \ref{sec2}, and deduce Szemer\'edi's theorem from it and the inverse theorem $\GI(s)$ in \S \ref{sec3}.

\section{Nilsequences are almost constant on progressions}\label{sec2}

In this paper the \emph{degree} $s$ of a nilsequence will not be particularly important, so we suppress most mention of it, recalling that it is nonetheless bounded by the complexity parameter $M$. The reader may care to note that, as a consequence of this, we do not need the full strength of $\GI(s)$ but only a weaker version in which correlation with a nilsequence of degree $O_s(1)$ (rather than $s$) is obtained. However, we know of no proof this result that is easier than the full-strength version and we also know (with Ziegler) a not especially painful argument for deducing the full version from the weak one.

The goal of this section is to prove Theorem \ref{recur-thm}. By induction on the dimension of the underlying nilmanifold, the result follows very quickly from the following.

\begin{proposition}\label{recur-prop}
Let the notation be as in Theorem \ref{recur-thm} above. Then we may partition $P$ into a disjoint union of progressions $P_i$, $i = 1,\dots, m$, each of length $\gg_{M,\eps}  |P|^{\kappa_M}$, and such that the following is true. For each $i = 1,\dots,m$ there is a polynomial nilsequence $(F_i(h_i(n)\Lambda_i))_{n \in \Z}$ of complexity $O_{M}(1)$ whose underlying nilmanifold has dimension strictly less than that of $(F(g(n)\Gamma))_{n \in \Z}$, and such that
\[ \diam_{P_i}(F(g(n)\Gamma) - F_i(h_i(n)\Lambda_i)) \leq \eps\] for all $i = 1,\dots, m$.
\end{proposition}

We derive this result in turn from three lemmas. The first and its proof are essentially \cite[Corollary 5.6]{gowers-longaps}, albeit formulated somewhat differently.  It can be viewed as an analogue of Theorem \ref{recur-thm} for polynomial phases.

\begin{lemma}[Polynomials are almost constant on progressions]\label{lem1}
Let $s \geq 1$ be an integer. Then there is some $\kappa_s > 0$ with the following property. Let $\phi : \R \rightarrow \R/\Z$ be a polynomial phase of degree $s$, and suppose that $P \subseteq \Z$ is a finite progression. Then we may partition $P$ into progressions $P_1,\dots, P_m$, $|P_i| \gg_s |P|^{\kappa_s}$, such that $\diam_{P_i}(\phi)  \leq \frac{1}{10}$ for $i = 1,\dots,m$. 
\end{lemma}
\begin{proof} Suppose that $|P|$ is sufficiently large in terms of $s$; the result is trivial otherwise, since we may partition into progressions of length 1. It suffices to prove the weaker statement that we may partition $P$ into progressions $P_1,\dots, P_m$, $|P_i| \gg_s |P|^{\kappa_s}$, such that for each $i$ there is a polynomial phase $\phi_i : \R \rightarrow \R/\Z$ of degree at most $s-1$ such that 
\[ \diam_{P_i}(\phi - \phi_i) \leq \frac{1}{100s^2}.\]  We may then work by induction on the degree to obtain the lemma (with a smaller value of $\kappa_s$ of course), using the fact that 
\begin{align*} \diam(\phi) \leq \diam (\phi - & \phi_s) + \diam   (\phi_s - \phi_{s-1}) + \dots  \\ & + \diam (\phi_1 - \phi_0) + \diam \phi_0  \leq \frac{1}{100}\sum_s \frac{1}{s^2} < \frac{1}{10}.\end{align*}

To obtain the weaker statement one invokes the following standard diophantine result essentially due to Weyl: there is some $\delta_s > 0$ such that, for any $\alpha \in \R/\Z$ and any $N \geq 1$, there is some $n \leq \sqrt{N}$ such that $\Vert \alpha n^s \Vert_{\R/\Z} \ll_s N^{-\delta_s}$.

Supposing that $\phi(n) = \theta n^s + \dots$ and that $P$ has common difference $d$ and length $N$, we apply this result with $\alpha := \theta d^s$. Subdividing $P$ into subprogressions $P_i$ of length between $N^{\delta_s/2}$ and $N^{2\delta_s/2}$ and common difference $dn$ gives the required statement. Note that such a subdivision is indeed possible since $n \leq \sqrt{N}$ and $N$ is sufficiently large in terms of $s$.\end{proof}

\begin{lemma}[Weyl-type equidistribution theorem]\label{lem2}
Suppose that $\phi : \R \rightarrow \R/\Z$ is a polynomial phase of degree $s$, and that $\diam_{[N]}(\phi) \leq \frac{1}{10}$. Then there is some $q = O_s(1)$ such that $\Vert q \phi \Vert_{C^{\infty}[N]} = O_s(1)$.
\end{lemma}
\begin{proof} This follows immediately from \cite[Proposition 4.3]{green-tao-nilratner} (the proof of which can be read independently of the rest of that paper, which we do not rely on heavily in this note). Observe, however, that it is quite classical and essentially goes back to Weyl, being the statement that a polynomial phase that is not equidistributed has almost rational coefficients. \end{proof}

\begin{lemma}[Factorisation of polynomial sequences]\label{lem3}
Let $(G/\Gamma, G_{\bullet})$ be a filtered nilmanifold of complexity $M$, and suppose that $g \in \poly(\Z, G_{\bullet})$. Let $\eta : G \rightarrow \R/\Z$ be a horizontal character with Lipschitz constant $O_M(1)$. Suppose that $P$ is an arithmetic progression and that $\diam_P (\eta \circ g) \leq \frac{1}{10}$. Then there is a factorisation $g = \beta g' \gamma$, where $\beta, \gamma \in \poly(\Z, G_{\bullet})$ and:
\begin{enumerate}
\item $\beta$ is smooth in the sense that $d_G(\beta(n),\beta(n')) = O_M(\delta)$ whenever $n,n' \in P$ and $|n - n'| \leq \delta |P|$;
\item $g'$ takes values in a connected $O_M(1)$-rational subgroup $G' \leq G$ with $\dim (G') < \dim (G)$;
\item $\gamma$ is $O_M(1)$-rational.
\end{enumerate}
\end{lemma}
\begin{proof} By rescaling linearly (and noting that if $g(n)$ lies in $\poly(\Z,G_{\bullet})$ then so does $g(an + b)$, cf. \cite[Lemma A.8]{green-tao-arithregularity}) we may assume that $P = [N]$. Applying 
Lemma \ref{lem2} and replacing $\eta$ by $\tilde \eta = q \eta$, where $q = O_M(1)$, we may assume that $\Vert \eta \circ g \Vert_{C^{\infty}[N]} = O_M(1)$.
The result may now be proved in exactly the same way as \cite[Proposition 9.2]{green-tao-nilratner} (although that result was a little more notationally intensive, formulated as it was for multiparameter sequences).
\end{proof}

\emph{Deduction of Proposition \ref{recur-prop}.} Select a nontrivial horizontal character $\eta : G \rightarrow \R/\Z$ with Lipschitz constant $O_M(1)$. Apply Lemma \ref{lem1} followed by Lemma \ref{lem3} to $\phi := \eta \circ g$, obtaining a decomposition of $P$ into progressions $P_i$ such that on each $P_i$ we have a factorisation $g = \beta g' \gamma$ of the stated type. Note that this factorisation depends on $i$, but we suppress this for notational convenience. Suppose that $\gamma(n)$ has period $q = O_M(1)$, so that $\gamma(n)\Gamma = \gamma(n')\Gamma$ whenever $n \equiv n' \md{q}$. Subdivide $P_i$ into progressions $P_j$  (this is a convenient abuse of notation) whose common difference is a multiple of $q$ and whose length is $c_{\eps,M} |P_i|$, for a constant $c_{\eps,M} > 0$ to be specified shortly. For each $j$, fix some $\gamma_0 = O_M(1)$ such that $\gamma(n)\Gamma = \gamma_0 \Gamma$ for all $n \in P_j$. Then if $n \in P_j$ we have
\[ g(n)\Gamma = \beta(n)\gamma_0 ( \gamma_0^{-1} g'(n)\gamma_0)\Gamma.\]
Set $H_j := \gamma_0^{-1} G' \gamma_0$ and $\Lambda_j := H_j \cap \Gamma$. Then $H_j/\Lambda_j$ is a nilmanifold of complexity $O_M(1)$, and certainly $\dim(H_j) = \dim (G') < \dim (G)$, and the polynomial sequence $h_j(n) := \gamma_0^{-1} g'(n)\gamma_0$ takes values in $H_j$. Pick some $n_0 \in P_j$, and define $F_j : H_j/\Lambda_j \rightarrow \C$ by
\[ F_j(x) := F(\beta(n_0)\gamma_0 x).\]
Then $F_j$ is $O_M(1)$-Lipschitz and
\begin{align*}
\diam_{P_j} (F(g& (n)\Gamma)  - F_j(h_j(n)\Lambda_j)) \\ & = \sup_{n \in P_j} |F(\beta(n) \gamma_0 ( \gamma_0^{-1} g'(n)\gamma_0)\Gamma) - F(\beta(n_0) \gamma_0 ( \gamma_0^{-1} g'(n)\gamma_0)\Gamma) | \\ & \leq \eps,
\end{align*}
the last line following if $c_{\eps,M}$ is sufficiently small from the smoothness of $\beta$ and the fact that $F$ has Lipschitz constant $O_M(1)$. 

\emph{Remark.} An almost identical argument appears in \cite[\S 2]{green-tao-ukmobius}. There, the reader will find a more careful discussion of the various rather rough assertions we have just made concerning Lipschitz constants and the like.

\section{Proof of Szemer\'edi's theorem}\label{sec3}

We now turn to the deduction of Szemer\'edi's theorem from Proposition \ref{recur-prop}.  As described in many places (for example \cite{gowers-longaps}) it follows easily by an iterated application of the following proposition.

\begin{proposition}[Density increment step]\label{dens-increment}
Suppose that $k$ is an integer and that $\alpha \in (0,1)$ is a parameter. Then there is a number $N_0(k,\alpha)$, a function $\omega_{k,\alpha} : \R^+ \rightarrow \R^+$ which tends to infinity and a non-decreasing function $\tau : (0,1) \rightarrow \R^+$ such that the following is true. Suppose that $P$ is a progression and that $A \subseteq P$ is a set of size $\alpha |P|$ containing no nontrivial $k$-term arithmetic progression. Then either $|P| \leq N_0(k,\alpha)$, or else there exists another arithmetic progression $P' \subseteq P$, $|P'|  \geq \omega_{k,\alpha}(|P'|)$, together with a set $A' \subseteq P'$ with $|A'| \geq (\alpha + \tau(\alpha))|P'|$ which contains no nontrivial $k$-term arithmetic progressions.
\end{proposition}

In applying this iteratively to establish Szemer\'edi's theorem, the point is that the second alternative can only occur $O_{\alpha}(1)$ times before the density of $A'$ inside $P'$ rises above 1, a contradiction.

\begin{proof} In this proof all implied constants are allowed to depend on $k$ and $\alpha$. By rescaling we may assume that $P = [N]$. Suppose then that $A \subseteq [N]$ is a set with cardinality $\alpha N$, but that $A$ contains no nontrivial $k$-term progressions. Define $f := 1_A - \alpha 1_{[N]}$ to be the balanced function of $A$, thus $\E_{n \in [N]} f(n) = 0$. Write 
\[ \Lambda_k (f_0,\dots, f_{k-1}) := \E_{n,d} f_0(n) f_1(n+d) \dots f_{k-1}(n+(k-1)d)\] for the multilinear operator counting $k$-term arithmetic progressions, and recall the generalised von Neumann theorem, which states that 
\[ |\Lambda_k(f_0,\dots, f_{k-1})| \ll \sup_{i = 0,\dots, k-1} \Vert f_i \Vert_{U^{k-1}}.\]
The expression $I := \Lambda_k (1_A, \dots, 1_A)$ is a normalised count of $k$-term progressions inside $A$, and we are supposing that the only such progressions are trivial (that is, have common difference 0). Therefore $I \leq 1/N$. On the other hand we may expand this as a sum of $2^k$ terms, the ``main'' term $\Lambda_k(\alpha,\dots,\alpha) = \alpha^k$ plus a sum of $2^k - 1$ other terms, each of which involves at least one copy of $f$. Supposing that $N > N_0(k,\alpha)$, the main term is much larger than the contribution of $1/N$ from the trivial progressions, and so one of these $2^k - 1$ other terms must be $\gg 1$. By the generalised von Neumann theorem this implies the crucial inequality
\[ \Vert f \Vert_{U^{k-1}} \gg 1.\]
By the inverse theorem for the Gowers $U^{k-1}$-norm (classical for $k = 3$, proved in \cite{green-tao-u3inverse} for $k = 4$, in \cite{green-tao-ziegler-u4inverse} for $k = 5$ and in the forthcoming paper \cite{uk-inverse} in the general case) this means that there is a degree $\leq (k-2)$ polynomial $1$-bounded nilsequence $(F(g(n)\Gamma))_{n \in \Z}$ of complexity $O(1)$ such that 
\[ |\E_n f(n) F(g(n)\Gamma)| \geq \delta,\] where $\delta \gg 1$.
Now we apply Theorem \ref{recur-thm} to partition $[N]$ into progressions $P_1,\dots,P_m$, each of length $\gg N^{c}$, such that $\diam_{P_i}(F(g(n)\Gamma)) \leq \delta/2$ for each $i$. Choose, for each $i$, some point $n_i \in P_i$. Then
\begin{align*} \delta N & = \sum_i \sum_{n \in P_i}  f(n) F(g(n_i)\Gamma)) + \sum_i \sum_{n \in P_i}  f(n) ( F(g(n)\Gamma) - F(g(n_i)\Gamma)) \\ & \leq \sum_i |\sum_{n \in P_i} f(n)| + \delta N/2,\end{align*}
and therefore $\sum_i |\sum_{n \in P_i} f(n)| \geq \delta N/2$.
Adding to this the equality \[\sum_i \sum_{n \in P_i} f(n) = 0\] and applying the pigeonhole principle, we conclude that there is at least one progression $P_i$ for which
\[ |\sum_{n \in P_i} f(n)| + \sum_{n \in P_i} f(n)  \geq \delta |P_i|/2,\] which means that
\[ \sum_{n \in P_i} f(n) \geq \delta |P_i|/4.\]
This means that the density of $A' := A \cap P_i$ in $P_i$ is at least $\alpha + \delta/4$, which implies Proposition \ref{dens-increment}.\end{proof}

\providecommand{\bysame}{\leavevmode\hbox to3em{\hrulefill}\thinspace}

\end{document}